\documentclass[10pt]{amsart}
\usepackage{calrsfs}

\usepackage[greek,english]{babel}
\usepackage[iso-8859-7]{inputenc}
\usepackage{graphicx}
\usepackage{hyperref}
\usepackage{amsthm}
\usepackage{amssymb}
\usepackage{multirow}
\usepackage{tikz-cd}
\usepackage{amsmath}
\usepackage{todonotes}
\usepackage{amsbsy}
\usepackage[all]{xy}
\usepackage{color, colortbl}
\definecolor{LightCyan}{rgb}{0.88,1,1}
\definecolor{Gray}{gray}{0.9}

\usepackage{changes}
\definechangesauthor[color=orange,name={Aristides Kontogeorgis}]{AK}

\usepackage[linenumbers]{MagmaTeX}
\newtheorem{theorem}{Theorem}
\newtheorem{lemma}[theorem]{Lemma}
\newtheorem{corollary}[theorem]{Corollary}
\newtheorem{proposition}[theorem]{Proposition}
\theoremstyle{definition}

\newtheorem{remark}[theorem]{Remark}
\newtheorem{definition}[theorem]{Definition}

\newcommand{\mdeg}{\rm mdeg}

\newcommand{\Z}{\mathbb{Z}}

\newcommand{\init}{{\rm in_\prec}}

\date{\today}

\title{Automorphisms and the canonical ideal}

\author[A. Kontogeorgis]{Aristides Kontogeorgis}
\address{Department of Mathematics, National and Kapodistrian  University of Athens
Pane\-pist\-imioupolis, 15784 Athens, Greece}
\email{kontogar@math.uoa.gr}

\author[A. Terezakis]{Alexios Terezakis }
\address{Department of Mathematics, National and Kapodistrian University of Athens\\
Panepistimioupolis, 15784 Athens, Greece}
\email{aleksistere@math.uoa.gr}

\author[I. Tsouknidas]{Ioannis Tsouknidas }
\address{Department of Mathematics, National and Kapodistrian University of Athens\\
Panepistimioupolis, 15784 Athens, Greece}
\email{iotsouknidas@math.uoa.gr}

\date \today

\makeatletter
\newcommand{\aprod}{\mathop{\operator@font \hbox{\Large$\ast$}}}
\makeatother

\begin{document}

\begin{abstract}
 The  automorphism group of a curve is studied from the viewpoint of the canonical embedding and Petri's theorem. A criterion for identifying the automorphism group as an algebraic subgroup the general linear group is given. Furthermore the action of the automorphism group is extended to an action of the minimal free resolution of the canonical ring of the curve $X$.  
\end{abstract}
\maketitle

%------------------------------
\section{Introduction}
%------------------------------

Let $X$ be a non-singular complete algebraic curve defined over an algebraically closed field of characteristic $p\geq 0$. If the genus $g$ of the curve $X$ is $g\geq 2$ then the automorphism group $G=\mathrm{Aut}(X)$ of the curve $X$ is finite. The theory of automorphisms of curves is an interesting  object of study, see the surveys \cite{Antoniadis2017}, \cite{MR3916740} and the references therein.

On the other hand the theory of syzygies which originates in the work of Hilbert and Sylvester has attracted a lot of researchers and it seems that a lot of geometric information can be found in the minimal free resolution of the ring of functions of an algebraic curve. For an introduction to this fascinating area we refer to \cite{MR2103875}. 

In this article we aim to put together the theory of syzygies of the canonical embedding and the theory of automorphism of curves. 
For this we assume that $X$ is a non-hyperelliptic, not trigonal and a non-singular quintic of genus $6$ and we also assume  $p\neq 2$. These conditions are needed for Petri's theorem to hold, while the $p\neq 2$ condition is needed to ensure the faithful action of the automorphism group on the space of holomorphic differentials $H^0(X,\Omega_X)$. 

More precisely, 
in section \ref{sec:algebraicEquations} we use Petri's theorem in order to  give a necessary and sufficient condition for an element in $\mathrm{GL}(H^0(X,\Omega_X))$ to act as an automorphism of our curve. In section \ref{sec:syzygies} we show that  the automorphism group  $G$ of the curve acts on a minimal free resolution $\mathbf{F}$ of the ring $S_X$. 
Notice that an action of a group $G$ on a graded module $M$ gives rise to a series of  linear representations $\rho_{d}:G \rightarrow M_d$ to all linear spaces $M_d$ of degree $d$ for $d\in \Z$.
For the case of the free modules $F_i$ of the minimal free resolution $\mathbf{F}$ 
we relate the actions of the group $G$ in both $F_i$ and in the dual $F_{g-2-i}$ in terms of an inner automorphism of $G$. 
This information is used in order to show that the action of the group $G$ on generators of the modules $F_i$ sends generators of degree $d$ 
to the set of generators of degree $d$. 

In the theory of syzygies in order to provide an invariant that does not depend on the selection of the minimal resolution, the $\mathrm{Tor}_i^S(k,S_X)$ is used, which is again a graded module. The degree $d$-part will be denoted by $\mathrm{Tor}_i^S(k,S_X)_d$,  which is a vector space of dimension $b_{i,d}$ and we prove that it is  also a $G$-module. 
 Moreover,
the representations to the $d$ graded space of
 each $F_i$,
$\rho_{i,d}:G \rightarrow \mathrm{GL}(F_{i,d})$ can be expressed as a direct sum of the $G$-modules $\mathrm{Tor}_i^S(k,S_X)_d$. We conclude by showing that the $G$-module structure of all $F_i$ is determined by knowledge of the $G$-module structure of $H^0(X,\Omega_X)$ and the $G$-module structure of each $\mathrm{Tor}_i^S(k,S_X)$ for all $0\leq i \leq g-2$. 

{\bf Acknowledgment:} The authors would like to thank K. Karagiannis for useful discussion concerning this article and his suggestions and corrections. The first author would like to also thank  G. Cornelissen for introducing him to Petri's theorem. 
I. Tsouknidas received financial support from  the Greek National scholarship foundation (IKY). 
%------------------------------
\section{Automorphisms of curves and Petri's theorem}
%------------------------------

Consider a complete non-singular non-hyperelliptic curve of genus $g\geq 3$ over an algebraically closed field $K$. 
%\comment[id=AK]{For deformation theory applications we need relative curves.}
Let $\Omega_X$ denote the sheaf of holomorphic differentials on $X$.
\begin{theorem}[Noether-Enriques-Petri]
There is a short exact sequence 
\[
0 \rightarrow I_X 
\rightarrow
\mathrm{Sym} H^0(X,\Omega_X) 
\rightarrow
\bigoplus_{n=0}^\infty H^0(X,\Omega_X^{\otimes n})
\rightarrow 
0,
\]
where $I_X$ is genereted by elements of degree $2$ and $3$. Also if $X$ is not a non-singular quintic of genus 6 or $X$ is not a trigonal curve, then $I_X$ is generated by elements of degree $2$. 
\end{theorem}
For a proof of this theorem we refer to \cite{Saint-Donat73}, \cite{MR895152}. The ideal $I_X$ is called 
{\em the canonical ideal} and it is the homogeneous ideal of the embedded curve $X \rightarrow \mathbb{P}_k^{g-1}$. The automorphism group of the ambient space $\mathbb{P}^{g-1}$ is known to be $\mathrm{PGL}_g(k)$, 
\cite[example 7.1.1 p. 151]{Hartshorne:77}. On the other hand every automorphism of $X$ is known to act on $H^0(X,\Omega_X)$  giving  rise to a representation 
\[
\rho: G \rightarrow \mathrm{GL}( H^0(X,\Omega_X)),
\] 
which is known to be faithful, when $X$ is not hyperelliptic and $p\neq 2$,   see \cite{MR3361016}. The representation $\rho$ in turn gives rise to a series of representations 
\[
\rho_d: G \rightarrow \mathrm{GL} (S_d), 
\] 
where $S_d$ is the vector space of degree $d$ polynomials in the ring $S:=k[\omega_1,\ldots,\omega_g]$. 
% \comment[id=AK]{Introduce Grothendieck rings and the computation for cyclic groups.}

Let $X\subset \mathbb{P}^r$ be a projective algebraic set. Is it true that every automorphism $\sigma:X \rightarrow X$ comes as the restriction of an automorphism of the ambient projective space, that is by an element of $\mathrm{PGL}_k(r)$. For instance such a criterion for complete intersections is explained in \cite[sec. 2]{Kontogeorgis2002-to}. In the case of canonical embedded curves $X\subset \mathbb{P}^{g-1}$ it is clear that any automorphism $\sigma\in \mathrm{Aut}(X)$ acts also in $\mathbb{P}^{g-1}=\mathrm{Proj} H^0(X,\Omega_X)$. In this way we arrive to the following
\begin{lemma}
Every automorphism $\sigma\in \mathrm{Aut}(X)$ corresponds to an element in $\mathrm{PGL}_g(k)$ such that $\sigma(I_X) \subset I_X$ and every element in $\mathrm{PGL}_g(k)$ such that $\sigma(I_X) \subset I_X$
gives rise to an automorphism of $X$.
\end{lemma}
In the  next section we will describe the elements  $\sigma \in \mathrm{PGL}_g(k)$ such that $\sigma(I_X) \subset I_X$.

%---------------------------
\subsection{Algebraic equations of automorphisms}
%---------------------------

\label{sec:algebraicEquations}

For now on we will assume that the canonical ideal $I_X$ is generated by polynomials in $k[\omega_1,\ldots,\omega_g]=\mathrm{Sym}H^0(X,\Omega_X) $ of degree $2$. 
Consider such a set of quadratic polynomials $\tilde{A}_1,\ldots,\tilde{A}_r$ generating $I_X$. 

A polynomial of degree two $\tilde{A}_i$  can be encoded in terms of a symmetric $g\times g$ matrix $A_i=(a_{\nu,\mu})$ as follows. Set $\bar{\omega}=(\omega_1,\ldots,\omega_g)^t$. We have
\[
\tilde{A}_i(\bar{\omega})=
\bar{\omega}^t A_i \bar{\omega}.
\]

Write $\sigma(A_i)$ for the symmetric $g\times g $ matrix such that $\sigma(\tilde{A}_i)=\bar{\omega}^t \sigma(A)_i\bar{\omega}$.

\begin{lemma}
For an element $\sigma\in \mathrm{GL}_{g}(k)$, $\sigma(I_X) \subset I_X$ holds if and only if for all $1\leq i \leq r$, $\sigma(A_i) \in \mathrm{span}_k \{A_1,\ldots,A_r\}$. 
\end{lemma}
\begin{proof}
$(\Leftarrow)$ If $\sigma(A_i) =\sum_j \lambda^i_j A_j$, then
\begin{align*}
& \bar{\omega}^t\sigma^tA_i\sigma \bar{\omega} = \sum_j \lambda_j^i \bar{\omega}^t A_j \bar{\omega} \Rightarrow \\
& \left(\sigma \bar{\omega}\right)^t A_i \left(\sigma \bar{\omega}\right) = \sum_j \lambda_j^i \tilde{A}_j
\end{align*}
Hence $\sigma(\tilde{A}_i) = \sigma(\bar{\omega}^t A_i \bar{\omega} ) \in I_X $.

\noindent$(\Rightarrow)$ If $\sigma(I_X)\subseteq I_X$, there are $g^i_j(X)\in S$ such that
\begin{equation*}
\sigma(\tilde{A}_i)=\sum_j g^i_j(X)\tilde{A}_j.
\end{equation*}
Since $\sigma$ respects the degree of the polynomials, $\mathrm{deg}(g^i_j(X))=0$ for all $i$, hence $g^i_j(X)=\lambda^i_j\in k$.

\begin{align*}
% \begin{split}
\bar{\omega}^t \sigma(A_i) \bar{\omega}
 &= \sigma (\bar{\omega}^t A_i \bar{\omega})= \sigma (\tilde{A}_i) = \\
= \sum_j \lambda^i_j \tilde{A}_j &= \sum_j \lambda^i_j \left( \bar{\omega}^t A_j \bar{\omega} \right) = \bar{\omega}^t \left( \sum_j \lambda^i_j A_j \right) \bar{\omega}
% \end{split}
\end{align*}
Since the matrices are symmetric it follows that $\sigma(A_i) \in \mathrm{span}_k\{\tilde{A_1},\dots,\tilde{A_r}\}$.
\end{proof}
The above lemma gives rise to a new representation
\begin{proposition}
\label{new-rep1}
Let $\sigma=(\sigma_{ij}) \in \mathrm{GL}_g(k)$. The element $\sigma$ gives rise to an automorphism of the curve with canonical ideal generated by $\tilde{A}_i, i=1,\ldots,r$ if and only 
\begin{equation} \label{actionI}
\sigma\left(
\tilde{A}_i
\right)
=\sum_{j=1}^r \lambda(\sigma)_{ji} \tilde{A}_j.
\end{equation}
% i.e. when the defining ideal is invariant. 
On the other hand $\sigma$ acts on $\bar{\omega}$ by 
\[
\sigma(\bar{\omega})=(\sigma_{\mu,\nu})\bar{\omega}
\]
therefore eq. (\ref{actionI}) can be written as
\[
(\sigma_{\mu,\nu})^t A_i (\sigma_{\mu,\nu})=\sum_{j=1}^r \lambda(\sigma)_{ji} A_j \qquad \text{ for every } 1\leq i \leq j.
\] 
\end{proposition}
Finding $g\times g$ matrices $(\sigma_{\mu,\nu})$ and elements $(\lambda_{i,j})$ so that the above equation holds will be called the {\em matrix automorphism problem}.

\subsection{The automorphism group as an algebraic set.}
Consider a set $A_1,\ldots,A_r$  of linear independent elements in the vector space of $g\times g$ symmetric matrices. For each invertible matrix $\sigma=(\sigma_{ij})\in \mathrm{GL}_g(k)$ we consider the $r$-elements $\sigma^t A_i \sigma$, $1\leq i  \leq r$. 
Fixing a basis for the space of $g\times g$ symmetric  matrices we can write any element $A_i$ as  $g(g+1)/2\times 1$ column matrix $\bar{A}_i$, that is 
\begin{align*}
\bar{\cdot}: \text{Symmetric } g\times g \text{ matrices} &\longrightarrow 
k^{\frac{g(g+1)}2} \\
A &\longmapsto \bar{A}
\end{align*}
We can now put together the $r$ elements $\bar{A}_i$ as a $g(g+1)/2 \times r$ matrix 
$\left( \bar{A}_1 | \cdots | \bar{A}_r \right)$, 
which  has full rank $r$, since $\{A_1,\ldots,A_r\}$ are assumed to be linear independent.

We have that $\sigma$ is an automorphism if the $g(g+1)/2\times 2r$-matrix
\[
B(\sigma)=
[\bar{A}_1,
\ldots,\bar{A}_r,
\overline{\sigma^t A_1\sigma},\ldots,\overline{\sigma^t  A_r \sigma}]
\]
has rank $r$, which means that $(r+1)\times (r+1)$-minors of $B(\sigma)$ are zero. 
This provides us with a description of the automorphism group as a determinantal variety. 

A simpler way to obtain algebraic equations, in order to see the automorphism curve as a finite algebraic group in $\mathrm{GL}_{g}(k)$ is using Gauss elimination to find a $\frac{g(g+1)}2\times \frac{g(g+1)}2$ invertible matrix  $Q$ which puts the matrix $\left( \bar{A}_1 | \cdots | \bar{A}_r \right)$ in echelon form, that is 
\[
Q \left( \bar{A}_1 | \cdots | \bar{A}_r \right)
=
\begin{pmatrix}
\mathbb{Id}_r \\
\hline
\mathbb{O}_{
\left(
\frac{g(g+1)}2-r
\right)\times r}
\end{pmatrix}.
\]
But then for each $1\leq i \leq r$ we have
\[
\overline{\sigma^t A_i \sigma} =
\sum_{j=1}^r \lambda_{j i} \bar{A}_i
\]
if and only if the lower $(\frac{g(g+1)}2 -r) \times r$ bottom block matrix of the matrix
\begin{equation}
\label{conditionAutomorphism}
\left(
\overline{\sigma^t A_1\sigma},\ldots,\overline{\sigma^t  A_r \sigma} 
\right)
\end{equation}
 is zero, while the top $r\times r$ block matrix gives rise to the representation
\[
\rho_1: G \rightarrow \mathrm{GL}_r(k), 
\]
defined in proposition \ref{new-rep1}.
Assuming that the lower $(\frac{g(g+1)}2 -r) \times r$ bottom block matrix gives us $r (\frac{g(g+1)}2 -r)$ equations where the entries $\sigma=(x_{ij})$ are seen as indeterminates. 
In this way we can write down elements of the automorphism group as a zero dimensional algebraic set, satisfying certain algebraic equations. 
% \comment[id=AK]{Use Noether normalization to reduce this set of equations to a simple polynomial equation. Define the field of definition of an automorphism}. 

\subsubsection{An example: Fermat curve}
Consider the projective non singular curve given by equation 
\[
F_n: x_1^n+x_2^n+x_0^n=0
\]
This curve has genus $g=\frac{(n-2)(n-1)}{2}$. 
Set $x=x_1/x_0$, $y=x_2/x_0$. 
For 
$\omega=\frac{dx}{y^{n-1}}=-\frac{dy}{x^{n-1}}$ we have that the set 
\begin{equation}
\label{Hol-Fermat}
x^i y^j \omega \text{ for } 0 \leq i+j \leq n-3
\end{equation}
forms a basis for holomorphic differentials, \cite{MR1643304}, \cite{Towse96}, \cite{MR2690176}.
These $g$ differentials are ordered lexicographically according to $(i,j)$, that is 
\[
\omega_{0,0} < \omega_{0,1} <\cdots < \omega_{0,n-3} < 
\omega_{1,0} < \omega_{1,1} < \cdots < \omega_{1,n-4} <
\cdots < \omega_{n-3,0}. 
\]
The case $n=2$ is a rational curve, the case $n=3$ is an elliptic curve, the case $n=4$ has genus $3$ and gonality $3$, the case $n=5$ has genus $6$ and is quintic so the first Fermat curve which has canonical ideal generated by quadratic polynomial is the case $n=6$ which has genus $10$. 
The appendix is devoted to the proof of the following
\begin{proposition}
\label{canGenerators}
The canonical ideal of the Fermat curve $F_n$ consists of two sets of relations
\begin{equation}\label{G1}
G_1=\{
\omega_{i_1,j_1} \omega_{i_2,j_2}- \omega_{i_3,j_3}
\omega_{i_4,j_4}: i_1+i_2 =i_3+i_4, j_1+j_2=j_3+j_4
\},
\end{equation}
and
\begin{equation}
\label{G2}
G_2=
\left\{
\omega_{i_1,j_1} \omega_{i_2,j_2}
+
\omega_{i_3,j_3} \omega_{i_4,j_4}
+
\omega_{i_5,j_5} \omega_{i_6,j_6}=0 
:
\substack{
	i_1+i_2=n+a,  \\
	i_3+i_4=a, \\
	i_5+i_6=a, 
}
\substack{
	j_1+j_2=b \\
j_3+j_4=n+b \\
j_5+j_6=b
}
\right\}
\end{equation}
where $0\leq a,b$ are selected such that $0\leq a+b \leq n-3$.
\end{proposition}
\subsubsection{Automorphisms of the Fermat curve}
The group of automorphisms of the Fermat curve in  is given by \cite{Tze:95},\cite{Leopoldt:96}
\[
G=
\begin{cases}
\mathrm{PGU}(3,p^h), & \text{ if }
n=1+p^h
\\
(\Z/n\Z \times \Z/n\Z) \rtimes S_3, &
\text{ otherwise }
\end{cases}
\]
The action of the automorphism group is given in terms of a $3\times 3$ matrix $A$ sending 
\[
x=(x_1/x_0) \mapsto 
\frac{
	\sum_{i=0}^2 a_{1,i}x_i
}{
	\sum_{i=0}^2 a_{0,i}x_i
}
\qquad
y=(x_2/x_0) \mapsto 
\frac{
	\sum_{i=0}^2 a_{2,i}x_i
}{
	\sum_{i=0}^2 a_{0,i}x_i
},
\]
In characteristic $0$, the matrix $A$ is a monomial matrix, that is, it has only one non-zero element in each row and column and this element is an $n$-th root of unity. Two matrices $A_1,A_2$ give rise to the same automorphism if and only if they differ by an element in the group $\{\lambda \mathbb{I}_3: \lambda \in k\}$. In any case the group $G$ is naturally a subgroup of $\mathrm{PGL}_3(k)$. Finding the representation matrix of $G$ as an element in $\mathrm{PGL}_{g-1}(k)$ is easy when $n\neq 1+p^h$ and more complicated in $n=1+p^h$ case. We have two different embeddings of the Fermat curve $F_n$ in projective space
\[
\xymatrix{
\mathbb{P}^{g-1}_k &	F_n \ar[r] \ar[l] & \mathbb{P}^2_k .
}
\]
In both cases the automorphism group is given as restriction of the automorphism group of the ambient space. 

The computation of the automorphism group in terms of the vanishing of the polynomials given in equation (\ref{conditionAutomorphism}) is quite complicated. The following program gives an idea of the complexity of this approach. The automorphism group for the $n=6$ case is described as an algebraic set described by $g^2=100$ variables and $756$ equations. 

> FermatCurve(6,Rationals());
> x_{7,8}*x_{10,10} - 2*x_{9,8}*x_{9,10} + x_{10,8}*x_{7,10},
>
>.................756 equations....................
>
> >x_{7,9}*x_{10,10} - 2*x_{9,9}*x_{9,10} + x_{10,9}*x_{7,10}

% where $(a,b)\in \Z/n\Z\times \Z/n\Z$ acts on the curve by 
% \[
% (a,b)x_1=\zeta_n^{a} \cdot x_1, (a,b) \cdot x_2= \zeta_n^{b} x_2. 
% \]
% The action of $\sigma\in S_3$ acts on the projective coordinates by permutation, that is 
% \[
% \sigma (X_0,X_1,X_2) = (X_{\sigma(0)},X_{\sigma(1)},X_{\sigma(3)}),
% \]
% so
% \[
% \sigma(x_1)=\sigma(X_1/X_0)=X_{\sigma(1)}/X_{\sigma(0)},
% \sigma(x_2)=\sigma(X_2/X_0)=X_{\sigma(2)}/X_{\sigma(0)}.
% \]
% For example the permutation $(0,1,2)$ acts on $x_1,x_2$ by 
% \[
% \sigma(x_1)=X_{2}/X_{1}=x_2/x_1, 
% \sigma(x_2)=X_0/X_1=1/x_1.
% \]
% If the characteristic of the field is $p>0$ and $n=p^h-1$ then the automorphism of the Fermat curve is 
% $\mathrm{PGU}(3,p^{2h})$, see \cite{Leopoldt:96}.

\section{Syzygies}

\label{sec:syzygies}

\subsection{Free resolutions}

Recall that $S=k[\omega_1,\ldots,\omega_g]$ is the polynomial ring in $g$ variables. Let $M$ be a graded $S$-module generated by the elements $m_1,\ldots,m_r$ of corresponding degrees $a_1,\ldots,a_r$. 
We consider the free $S$-module
$
F_0=\bigoplus_{j=1}^r S(-a_j)
$ 
together with the onto map 
\[
F_0=\bigoplus_{j} S(-a_j) 
\stackrel{\pi}{\longrightarrow}
M.
\]
Let us denote by $M_1,\ldots,M_r$ elements of $F_0$, 
such that $\pi(M_i)=m_i$, assuming also that $\deg(M_i)=\deg(m_i)$,
 for $1\leq i \leq r$.  
In the above formula we have used the isomorphism 
\begin{align}
\label{graded-isomorphism}
\bigoplus_{j=1}^r S(-a_j) 
& \longrightarrow \bigoplus_{j=1}^r M_i S 
\\
(s_1,\ldots,s_r) 
& \longmapsto
\sum_{j=1}^s s_i M_i. \nonumber
\end{align}
We introduced the shifts in the grading so that 
$\deg_{S(-a_j)} s_j =\deg_M (s_j M_j)$, that is 
$\deg_{S(-a_j)} s_j =\deg_S (s_j) +\deg(M_j)=\deg_S (s_j)+a_j$.

The kernel of the map $\pi$ is again a finite generated $S$-module and by continuing this operation we arrive at a free resolution of $M$, that is,
at a sequence of free $S$-modules
\begin{equation} \label{free-resolution}
\xymatrix{
	0 \ar[r] & 
	F_g \ar[r]^{\phi_g} &
	\cdots \ar[r] &
    F_1 \ar[r]^{\phi_1} \ar[r] &
    F_0
},
\end{equation}
where $\mathrm{coker}(\phi_1)=F_0/\mathrm{Im}\phi_1=F_0/\mathrm{ker}\pi \cong M$. 

Let $\mathfrak{m}$ be the maximal ideal of $S$ generated by $\langle \omega_1,\ldots,\omega_g \rangle$. The pair $(S,\mathfrak{m})$ behaves in many aspects like a local ring, see \cite{MR494707}. 
The graded free resolution in eq. (\ref{free-resolution}) is called minimal if for each $i$ the image of $\phi_i$ is contained in $\mathfrak{m}F_{i-1}$ or equivalently if the reduced maps $\bar{\phi_i}: F_{i}/\mathfrak{m}F_i \rightarrow F_{i-1}/\mathfrak{m} F_{i-1}$ are zero. From now on all free resolutions, we will use, are assumed to be minimal. 

Each free module in the resolution can be written as 
\[
F_i=\bigoplus_j S(-j)^{\beta_{i,j}}.
\]
The integers $\beta_{i,j}$ are called the Betti numbers of the resolution. The Betti diagram of the free resolution is given by the following table:
\begin{center}
\newcolumntype{g}{>{\columncolor{Gray}}c}
\begin{tabular}{|g|c|c|c|c|}
\hline
\rowcolor{LightCyan}
 & 0 &
1  &
$\cdots$ &
r\\
\hline
$i$ & $\beta_{0,i}$ & $\beta_{1,i+1}$  & $\cdots$  & $\beta_{r,i+r},$ \\
$i+1$ & $\beta_{0,i+1}$ & $\beta_{1,i+2}$ & $\cdots$ & $\beta_{r,i+r+1}$\\
$\vdots$ & $\vdots$ & $\vdots$ &   & $\vdots$\\
j& $\beta_{0,j}$ & $\beta_{1,j+1}$ & $\cdots$  & $\beta_{r,j+r}$  \\
\hline
\end{tabular}
\end{center}
For example the Fermat curve for $n=6$, that is $g=10$ gives rise to the following Betti diagram, as we can compute using the magma algebra system, \cite{Magma1997}:
\begin{center}
\newcolumntype{g}{>{\columncolor{Gray}}c}
\begin{tabular}{|g|c|c|c|c|c|c|c|c|c|}
\hline
\rowcolor{LightCyan}
 & 0 &
1  & 2 & 3 & 4  & 5 & 6 &  7  &8 \\
\hline
0 & 1 & 0 & 0 & 0 & 0 & 0 & 0 & 0 & 0 \\
1& 0 & 28 & 105 & 189 & 189& 105 & 27 & 0 & 0 \\
 2& 0 & 0 & 27 & 105 & 189 & 189 & 105 & 28 & 0 \\
3 & 0 & 0 & 0 & 0 & 0 & 0 & 0 & 0 & 1 \\
\hline
\end{tabular}
\end{center}
We observe that the above table is highly symmetrical and  this is not a coincidence. Observe also  that only the first three rows are non zero and  this is due to a general result stating that the Casteluovo-Mumford regularity for the ring $S/I_X$ is three, see \cite[9B]{MR2103875}. Also 
%\highlight[id=AK,comment=Eisenbud writes that this should be $\binom{g-1}{2}$, 
%but perhaps this is a typo.
%]{
$28=\beta_{1,2}=\binom{g-2}{2}$.
%}.
%YES he correctd in the printed version of the book
The evident symmetry in the above matrix (top to down, left to right) is part of a general theory as we will see in section \ref{sec:Gorenstein}. 
We will introduce group actions on free resolutions first. 
%%%%%%%%%%%%%%%%%%%%%%%%%%%%%%%%%%%%%%%
\subsection{Extending group actions}
%%%%%%%%%%%%%%%%%%%%%%%%%%%%%%%%%%%%%%%

Let $M$ be a finitely generated graded $S$-module acted on by the group $G$. Let $m_1,\ldots,m_r$ be minimal  generators of $M$ of corresponding degrees $a_i$. Every element $\sigma\in G$ sends every generator $M_i$ of the free module $F_0$
\begin{equation}
\label{matrixA}
\sigma(M_i)= \sum_{\nu=1}^r a_{\nu,i} M_i, \text{ for some } a_{\nu,i} \in S. 
\end{equation}
\begin{remark}
We would like to point out here that unlike the theory of vector spaces, an element $x\in F_0$ might admit two different decompositions
\[
x=\sum_{i=1}^r a_i m_i=\sum_{i=1}^r b_i m_i,
\]
that is 
\[
\sum_{i=1}^r (a_i-b_i) m_i=0, 
\]
and if $a_{i_0}-b_{i_0}\neq 0$ we cannot assume that it is invertible, so we can't express $M_{i_0}$ as an 
$S$-linear combination of the other elements $M_i$, for $i_0\neq i, 1\leq i \leq r$. We can only deduce that $\{a_i-b_i\}_{i=1,\ldots,r}$ form a syzygy. 

Therefore one might ask if the matrix expression given in eq. (\ref{matrixA}) is unique.  In proposition \ref{prop:respect-degrees} we will prove that the elements $a_{\nu,i}$ which appear as coefficients in eq. (\ref{matrixA}) are in the field $k$ and therefore the expression is indeed unique. 
\end{remark}

Observe that we have extended the natural action of $\mathrm{Aut}(X)$ on $H^0(X,\Omega_X)$ to an action on the ring $S=\mathrm{Sym}H^0(X,\Omega_X)$, so that 
$\sigma(xy)=\sigma(x)\sigma(y)$. Therefore if $M=I_X$ then for all $s\in S$, $m\in I_X=M$ we have $\sigma(s m)=\sigma(s) \sigma(m)$. All the actions in the modules we will consider will have this property.  

We can extend the action of $G$ on the free $S$-module 
\[
F_0=M_1 S \oplus \cdots \oplus M_r S
\]
by the rule $\sigma(s M_i)=\sigma(s) \sigma(M_i)$.
As we did in the isomorphism in eq. (\ref{graded-isomorphism}) we would like to transfer the action to the direct sum of the shifted ring 
$\bigoplus_{j=1}^r S(-a_j)$. 
Observe that 
\[
\sigma
\left(
\sum_{j=1}^r s_j M_j
\right)
=
\sum_{j=1}^r \sigma(s_j) \sum_{\nu=1}^r a_{\nu,j}(\sigma) M_\nu=
\sum_{\nu=1}^r  
\left(
\sum_{j=1}^r a_{\nu,j}(\sigma) \sigma(s_j)
\right)
M_\nu,
\]
where $\deg_S a_{\nu,j}+a_\nu=\deg_S m_j$. This means that under the action of $\sigma\in G$ the $r$-tuple 
$(s_1,\ldots,s_r)^t$ is sent to
\[
\begin{pmatrix}
s_1 \\
\vdots \\
s_r
\end{pmatrix}
\stackrel{\sigma}{\longmapsto}
\begin{pmatrix}
a_{1,1}(\sigma) & a_{1,2}(\sigma) & \cdots & a_{1,r} (\sigma)
\\
\vdots & \vdots & & \vdots 
\\
a_{r,1}(\sigma) & a_{r,2}(\sigma) & \cdots & a_{r,r}(\sigma)
\end{pmatrix}
\begin{pmatrix}
\sigma(s_1) 
\\
\vdots
\\
\sigma(s_r)
\end{pmatrix}.
\] 
If $A(\sigma)=\big( a_{i,j}(\sigma) \big)$ is the matrix corresponding to $\sigma$ then for $\sigma,\tau\in G$ the following cocycle condition holds:
\[
A(\sigma \tau)=A(\sigma) A(\tau)^{\sigma}.
\]
If we can assume that $G$ acts trivially on the  matrix $A(\tau)$ for every $\tau\in G$ (for instance when $A(\tau)$ is a matrix with entries in $k$ for every $\tau\in G$), then the above cocycle condition becomes a homomorphism condition. 

Also if $A(\sigma)$ is a principal derivation, that is there is an $r\times r$ matrix $Q$, such that 
\[
A(\sigma)=\sigma(Q)\cdot Q^{-1}
\]
then after a basis change of the generators we can show that the action on the coordinates is just given by 
\[
(s_1,\cdots,s_r)^t
\stackrel{\sigma}{\longmapsto}
(\sigma(s_1),\cdots,\sigma(s_r))^t,
\]
that is the matrix  $A(\sigma)$ is the identity. 
We will call the action on the free resolution $\mathbf{F}$ obtained by extending the action on $M$ the standard action.

%%%%%%%%%%%%%%%%%%%%%%%
\subsection{Group actions on free resolutions}
%%%%%%%%%%%%%%%%%%%%%%%
\label{sec:actionOnFreeResolution}
Assume that $M$ is acted on by a group $G$ and let 
\begin{equation}
\label{free-resolution1}
0\rightarrow F_n \stackrel{\delta_n}{\longrightarrow}
F_{n-1} \rightarrow \cdots F_i 
\stackrel{\delta_i}{\longrightarrow}
F_{i-1}
\rightarrow 
\cdots
F_1
\rightarrow 
M 
\rightarrow 
0
\end{equation}
be a minimal free resolution, that is each $F_i$ is a free graded $S$-module and the image of each $\delta_i$ is in $\mathfrak{m} F_{i-1}$, or equivalently $\delta_i$ maps the basis of $F_i$ to a minimal set of generators of the image of $\delta_i$, see \cite[sec. 1B]{MR2103875}. The image of each $\delta_i$ is the $i$-th module of syzygies. 

Assume that each $F_i$ is acted on by $G$ and that the maps $\delta_i$ are $G$-equivariant. 
Then a series of representations can be defined:
\begin{equation}
\label{TorRep}
\rho_i: G \rightarrow 
\mathrm{GL}(\mathrm{Tor}_i^S(k,M)).
\end{equation}
Indeed, from the resolution given in eq. (\ref{free-resolution}) we can consider the complex 
\begin{equation} \label{tensor-resolution}
\cdots \rightarrow k\otimes F_i 
\stackrel{1_k \otimes \delta_i}{\longrightarrow}
k \otimes F_{i-1} 
\rightarrow 
\cdots
\end{equation}
Since the $\delta_i$ maps are zero modulo $\mathfrak{m}$ all maps in the above complex are zero. On the other hand the homology of the complex (\ref{tensor-resolution}) is by definition $\mathrm{Tor}_i^S(k,M)$ so $\mathrm{Tor}_i^S(k,M)=k\otimes F_i$. 

We will now  study  the  action of the group $G$ on the generators of $F_i$. First of all we have that 
\[
F_i=\bigoplus_{\nu=1}^{r_i} \bigoplus_{\mu=1}^{b_{i,\nu}} 
e_{i,\nu,\mu} S \cong
\bigoplus_{\nu=1}^{r_i} 
% \bigoplus_{\mu=1}^{b_{i,\nu}}
 S(-d_{i,\nu})^{b_{i,\nu}}.
\]
In the above formula we assumed that $F_i$ is generated by elements $e_{i,\nu,\mu}$ such that the degree of $e_{i,\nu,\mu}=d_{i,\nu}$ for all $1\leq \mu \leq b_{i,\nu}$. We also assume that 
\[
d_{i,1} < d_{i,2} < \cdots < d_{i,r_i}. 
\]
The action of $\sigma$ is respecting the degrees, so an element of minimal degree $d_{i,1}$ is sent to a linear combination of elements of minimal degree $d_{i,1}$. In this way we obtain a representation 
\[
\rho_{i,1}: G \rightarrow \mathrm{GL}(b_{i,1},k). 
\]
In a similar way an element $e_{i,2,\mu}$ of degree $d_{i,2}$ is sent to an element of degree $d_{i,2}$ and we have that 
\[
\sigma(e_{i,2,\mu})=
\sum_{j_1=1}^{b_{i,2}} \lambda_{i,2,\mu,j_1} e_{i,2,j_1}
+
\sum_{j_2=1}^{b_{i,1}} \lambda'_{i,2,\mu,j_1} e_{i,1,j_2},
\]
where all $\lambda_{i,2,\mu,j_1}\in k$
and all $\lambda'_{i,1,\mu,j_2} \in S/\mathfrak{m}^{d_{i,2}-d_{i,1}+1}$. 
In this case we have a representation with entries in an Artin algebra instead of a field, which has the form:
\[
\rho_{i,2}:G \rightarrow \mathrm{GL}(b_{i,1}+b_{i,2},
S/\mathfrak{m}^{d_{i,2}-d_{i,1}+1}),
\]
\[
\sigma \mapsto 
\begin{pmatrix}
A_1(\sigma) & A_{1,2}(\sigma) \\
0 & A_2(\sigma)
\end{pmatrix},
\]
where $A_1(\sigma)\in \mathrm{GL}(b_{i,1},k)$ and $A_2(\sigma)\in \mathfrak{m}^{d_{i,2}-d{i,1}}\mathrm{GL}(b_{i,2},k)$. 

By induction the situation in the general setting gives rise to a series of nested representations:
\[
\rho_{i,j}: G \rightarrow \mathrm{GL}(b_{i,1}+b_{i,2},
S/\mathfrak{m}^{d_{i,j}-d_{i,1}+1})
\]
\begin{equation}
\label{representation-matrix}
\sigma \mapsto 
A(\sigma)=
\begin{pmatrix}
A_1(\sigma) &  A_{1,2}(\sigma) & \cdots & A_{1,j}(\sigma) \\
0 & A_2(\sigma) &  & A_{2,j}(\sigma) \\
\vdots & \ddots &  & \vdots \\
0 & \cdots & 0 & A_{j}(\sigma)
\end{pmatrix}
\end{equation}
where $A_\nu(\sigma) \in \mathrm{GL}(b_{i,\nu},k)$
and $A_{\kappa,\lambda}(\sigma)$ is an 
$b_{i,\kappa} \times b_{i,\lambda}$ matrix with coefficients in $\mathfrak{m}^{b_{i,\lambda}-b_{i,\kappa}}$.
The representation $\rho_{i,r_i}$ taken modulo $\mathfrak{m}$ reduces to $\mathrm{Tor}_i^S(k,M)$, seen as a $k[G]$-module. 

\begin{remark}
Using the same construction as in the Maschke theorem, see \cite[p.13]{alperin} we can prove that this representation is equivalent to a block diagonal representation, if the characteristic $p \nmid |G|$.
We will see in proposition \ref{prop:respect-degrees} that under mild hypotheses this requirement can be removed.  
\end{remark}

%-----------
\subsection{Gorenstein symmetry}
%-----------
\label{sec:Gorenstein}
For the canonical embedding $X \hookrightarrow \mathbb{P}^{g-1}$ we have $r=g-1$,  therefore $S(-r-1)=S(-g)$. 
The $S$-module $S(-r-1)$ is special, for instance since $X$ is smooth we have by \cite[th. 3.1]{MR0274461} that the sheaf corresponding to the module $S(-g)$ is canonically isomorphic to $\wedge^{g-1} \Omega^{1}_X$ thus it is naturally acted by the automorphism group of $X$.

If $\mathbf{F}$ is a free resolution
\[
\mathbf{F}:
\xymatrix{0 \ar[r] &  F_r
\ar[r]^{\phi_r} &
\cdots
\ar[r]^{\phi_{i+1}} &
 F_i \ar[r]^{\phi_i} &
 F_{i-1} \ar[r]^{\phi_{i-1}} &
\cdots \ar[r] &
F_1 \ar[r]^{\phi_1} & 
F_0
 }
\]
 of $S/I_X$ then we twist  it by $S(-r-1)$ in order to obtain a new free resolution
  $\mathrm{Hom}_S(F_i,S(-g))$ of $S/I_X$:
\[
\mathbf{F}^*:
% \mathbf{F}:
\xymatrix{0  & \ar[l]   F_r
 &\ar[l]_{\phi_r^*} 
\cdots
& \ar[l]_{\phi_{i+1}^*} 
 F_i  & \ar[l]_{\phi_i^*} 
 F_{i-1} & \ar[l]_{\phi_{i-1}^*} 
\cdots & \ar[l] 
F_1 & \ar[l]_{\phi_1^*} 
F_0
 }
\]
The ring $S_X:=S/I_X$ is Cohen-Macauley of codimension $r-1=g-2$ therefore
\[
\mathrm{Ext}_S^{i}(S_X,S(-g))=0 \text{ for } i\neq g-2.
\]
This means that $\mathbf{F}$ is a free resolution of 
$\mathrm{Ext}_S^{g-2}(S_X,S(-g))$, which in turn is isomorphic to $S_X(1)$, see \cite[prop. 9.5]{MR2103875}.
\begin{lemma}
\label{Homdual}
We have 
\[
\mathrm{Hom}_S (S(-a),S(-b)) \cong S(-b+a).
\]
\end{lemma}
\begin{proof}
There are elements $e_a,e_b$ of degrees $a,b$ respectively
such that $e_aS\cong S(-a)$ and $e_b S \cong S(-b)$, where the $S$-module isomorphisms are given by sending  $e f \mapsto f$. This means that $\deg_{e_aS}( e_af )= \deg_{S(-a)} f$ and $\deg_{e_bS}( e_bf )= \deg_{S(-b)} f$. In particular $\deg_{S(-a)} 1_S=a$ or in other words $S(-a)_d=S_{-a+d}$, see also \cite[prop. 2.3]{MR2560561}.

The set $\mathrm{Hom}_S(N,M)$ where $N,M$ 
are graded $S$-modules with $N$ finitely generated equals $\bigoplus_{i\in \Z} \mathrm{Hom}_i(N,M)$, where $\mathrm{Hom}_i(N,M)$ consists of all morphisms of degree $i$, see \cite[prop. 2.7]{MR2560561}.  

A homomorphism $\phi$ satisfies $\phi(e_a f)=f \phi(e_a)$ and 
 is described completely by $\phi(e_a)\in e_b S$. 
If $\phi$ is of degree $0$ we  furthermore 
 have 
$\phi(e_a)=e_b F_\phi$, where the degree of  $F_\phi$
satisfies the equation $a=b+\deg_S(F_\phi)$, that is $\deg(F_\phi)=a-b$.
This proves that 
$\mathrm{Hom}_{0}(S(-a),S(-b))\cong  S_{a-b}$, thus $\mathrm{Hom}_S(S(-a),S(-b))\cong S(-b+a)$.  
\end{proof}
\begin{corollary}
For $b=g$ we have
\[
\mathrm{Hom}_S( S(-a),S(-g))=S(-g+a)
\]
\end{corollary}
Since $F^*$ is a free resolution of $S_X(1)$ (shifted by one) we arrive at the symmetry
\begin{equation}
\label{BettiSymmetry}
\beta_{i,j}=\beta_{g-2-i,g+1-j}.
\end{equation}
This symmetry can also be interpreted in terms of  of Koszul cohomology, see \cite[prop. 4.1]{MR3729076}.
In general the Betti table for a canonical model is given by
\begin{center}
\newcolumntype{g}{>{\columncolor{Gray}}c}
\begin{tabular}{|g|c|c|c|c|c|c|c|c|c|c|c|}
\hline
\rowcolor{LightCyan}
 & 0 & 1  & $\cdots$ & $a$ & $a+1$  & $\cdots$ & $b-1$ &  $b$  & $\cdots$ & $g-3$ & $g-2$ \\
\hline
0 & 1 & 0 & $\cdots$ & 0 & 0 & $\cdots$ & 0 & 0 &  $\cdots$ & 0 & 0 \\
1& 0 & $\beta_1$ & $\cdots$ & $\beta_a$ & 
$\beta_{a+1}$ & $\cdots$ & $\beta_{g-3-a}$& 0 & $\cdots$ & 0 & 0 \\
 2& 0 & 0 & $\cdots$ & 0 & $\beta_{g-3-a}$ & $\cdots$ & $\beta_{a+1}$ & $\beta_a$ & $\cdots$ & $\beta_1$ &  0 \\
3 & 0 & 0 & $\cdots$ & 0 & 0 &$\cdots$  & 0& 0 & $\cdots$&0  & 1 \\
\hline
\end{tabular}
\end{center}
In the above table the $i,j$-entry corresponds to $\beta_{i,i+j}$. 
The Green conjecture states that the integer $a$
is equal to $\mathrm{Cliff}(X)-1$. Notice that the  Green conjecture is known to fail in positive characteristic, see \cite{1803.10481}.

\subsection{Unique actions}
Let us consider two actions of the automorphisms group $G$ on $H^0(X,\Omega_X)$, which can naturally be extended on the symmetric algebra
 $\mathrm{Sym}H^0(X,\Omega_X)$. We will denote the first action by $g\star v$ and the second action by $g \circ v$, where $g\in G$, $v\in \mathrm{Sym}H^0(X,\Omega_X)$.
\begin{proposition}
\label{prop:two-actions-by-iso}
If the curve $X$ satisfies the conditions of faithful action of $G=\mathrm{Aut}(X)$ on $H^0(X,\Omega_X)$, that is $X$ is not hyperelliptic and $p>2$, \cite[th. 3.2]{MR3361016} and moreover both actions $\star,\circ$
restrict to actions on the canonical ideal $I_X$, then there is an automorphism $i:G\rightarrow G$, such that 
$g\star v=i(g) \circ v$. 
\end{proposition}
\begin{proof}
Both actions of $G$ on $H^0(X,\Omega_X)$ introduce automorphisms of the curve $X$. That is since $G\star I_X=I_X$ and $G\circ I_X=I_X$, the group $G$ is mapped into $\mathrm{Aut}(X)=G$. This means that for every element $g\in G$ there is an element $g^*\in \mathrm{Aut}(X)=G$ such that 
$g\star v=g^*v$, where the action on the right is the standard action of the automorphism group on holomorphic differentials. By the definition of the group action  for every $g_1,g_2\in G$ we have $(g_1g_2)^*v=g_1^* g_2^*v$ for all $v\in H^0(X,\omega_X)$ and the faithful action of the automorphism group provides us with $(g_1g_2)^*=g_1^* g_2^*$, i.e. the map $i_*:g\mapsto g^*$ is a homomorphism. Similarly the map corresponding to the $\circ$-action,  $i_\circ:g\mapsto g^\circ$ is a homomorphism and the desired homomorphism $i$ is the composition of $i_* i_\circ^{-1}$. 
\end{proof}
So far we have introduced the map $\mathrm{Hom}_S(F_i,S(-g))$ which induces a symmetry of the free resolution $\mathbf{F}$ by  sending  $F_i$ to $F_{g-2-i}$. 
Then each free module $F_i$ of the resolution  $\mathbf{F}$ is equipped by the extension of the action on holomorphic differentials,  according to the construction of section \ref{sec:actionOnFreeResolution}. On the other hand since $S(-g)$ is a $G$-module we have that $F_{g-2-i}\cong\mathrm{Hom}_S(F_i,S(-g))$ is equipped by a second action namely every $\phi: F_i\rightarrow S(-g)$ is acted naturally by $G$ in terms of $\phi\mapsto \phi^\sigma=\sigma^{-1} \phi \sigma$. 
How are the two actions related?
\begin{lemma}
\label{star-dual-action}
Denote by $\star$ the action of $G$ on $F_i$ induced by taking the $S(-g)$-dual. The standard and the $\star$-actions are connected in terms of an automorphism $\psi_i$ of $G$, that is for all $v\in F_i$ $g\star v= \psi_i(g) v$. 
\end{lemma}
\begin{proof}
Assume that $i \leq g-2-i$. Consider the standard action of $G$ on the free resolution $\mathbf{F}$. The module $F_{g-2-i}$ obtains a new action $g\star v$ for $g\in G, v\in F_i$. By \ref{sec:actionOnFreeResolution}
this $\star$ action is transfered to an action on all $F_j$ for $j\geq g-2-i$, including the final term $F_{g-2}$ which is isomorphic to $S(-1)$. This gives us two actions on $H^0(X,\Omega_X)$ which satisfy the requirements of proposition \ref{prop:two-actions-by-iso}. The desired result follows.  
\end{proof}

\begin{proposition}
\label{prop:respect-degrees}
Under the faithful action requirement we have that all automorphisms $\sigma\in G$ send the direct summand
 $S(-j)^{\beta_{i,j}}$ of $F_i$ to itself, that is the  representation matrix in eq. (\ref{representation-matrix}) is block diagonal. 
\end{proposition}
\begin{proof}
Consider $F_i=\bigoplus_{\nu=1}^{r_i} M_{i,\nu} S$, where $M_{i,1_i},\ldots, Μ_{i,r_i}$ are assumed to 
be  minimal generators of $F_i$ with descending degrees $a_{i,\nu}=\deg(m_{i,\nu})$, $1\leq \nu \leq r_i$.  
The action of an element $\sigma$ is given in terms of the matrix $A(\sigma)$ given in equation (\ref{representation-matrix}). The element $\phi\in \mathrm{Hom}_S(F_i, S(-g))$ is sent to 
\begin{align*}
h:\mathrm{Hom}_S(F_i, S(-g)) & 
\stackrel{\cong}{\longrightarrow}
 F_{g-2-i} 
\\
\phi 
  & \longmapsto  
\big(
\phi(M_{i,1}), \ldots,\phi(M_{i,r_i})
\big)
\end{align*}
Each $\phi(M_{i,\nu})$ can be considered as an element in $S(-g-1+\deg(m_{i,\nu}))$ inside $F_{g-2-i}$.
Observe that the element $\phi \in \mathrm{Hom}_S(F_i, S(-g))$ is known if we know all $\phi(M_i,\nu)$ for 
$1\leq \nu \leq r_i$. From now on we will identify  such an  element $\phi$ as a $r_i$-tuple
  $\big(\phi(M_i,\nu)\big)_{1\leq \nu \leq r_i}$. 
 We can consider as a basis of $\mathrm{Hom}(F_i,S(-g))$ the morphisms $\phi_\mu$ given by 
\[
\phi_\mu(M_j)=\delta_{\mu,j}\cdot E,  
\]
where $E$ is a basis element of degree $g$ of the rank 1 module $S(-g)\cong S\cdot E$. 
This is a different basis than the basis $M_{g-2-i,\nu}$, $1\leq n \leq r_{g-2-i}$ of $F_{g-2-i}$ we have already introduced.

Recall that if $A,B$ are $G$-modules, then there is an natural action on $\mathrm{Hom}(A,B)$, sending $\phi\in \mathrm{Hom}(A,B)$ to $^\sigma\phi$ which is the map 
\[
^\sigma\phi:A \ni a \mapsto \sigma \phi(\sigma^{-1}a) 
\]

We have also a second action on the module $F_{g-2-i}$. We compute $^\sigma\phi(M_{i,\nu})$ for all base elements $M_{i,\nu}$ in order to descrive $^\sigma\phi	$:
\begin{align*}
\sigma \left(
\phi( \sigma^{-1}  M_{i,\nu})
\right)_{
1\leq \nu \leq  \kappa 
	}
&=
\left(
\sum_{\mu=1}^{r_i}
\sigma
\big(\alpha_{\mu,\nu}(\sigma^{-1}) 
\big)
\sigma\phi(M_{i,\mu})
 \right)_{
1\leq \nu \leq  r_i
	}
\\
& =
\left(
\sum_{\mu=1}^{r_i}
\sigma\big(\alpha_{\mu,\nu}(\sigma)\big) \chi(\sigma)\phi(M_{i,\mu})
 \right)_{
1\leq \nu \leq  r_i
	}
\end{align*}
where in the last equation we have used the fact that 
$\phi(M_i)$ are in rank one $G$-module $S(-g)\cong \wedge^{g-1} \Omega^{1}_X$ hence the action of $\sigma\in G$ is given by multiplication by $\chi(\sigma)$, where $\chi(\sigma)$ is an invertible element is $S$.

In order to simplify the notation consider $i$ fixed, and denote $M_\nu=M_{i,\nu}$, $r=r_i$, $a_{i,j}=a_j$. 

According to eq. (\ref{BettiSymmetry}) if $M_j$ has degree $a_j$ then the element $\phi_j$ has degree $g+1-a_j$. 
Assume that $M_r$ has maximal degree $a_r$. Then, $\phi_r$ has minimal degree.
Moreover, in order to describe $^\sigma\phi_r$ we have to consider
the tuple 
$(^\sigma\phi_r(M_1),\ldots,^\sigma\phi_r(M_r))$. We have   
\begin{align*}
\big(
^\sigma\phi_r(M_\nu)
\big)_{1\leq \nu \leq r}
&=
\left(
\sum_{\mu=1}^r 
\sigma
\big(\alpha_{\mu,\nu}^{(i)}(\sigma^{-1}) 
\big)
\chi(\sigma) \phi_r(M_\mu)
\right)_{
1\leq \nu \leq  r
	}
	\\
	&=
	\left(
\sigma
\big(
\alpha_{r,\nu}^{(i)}(\sigma^{-1})
\big)
 \chi(\sigma) E
\right)_{
1\leq \nu \leq  r
	}
\end{align*}
and we finally conclude that
\[
^{\sigma}\phi_r
=
\sum_{\nu=1}^r 
\sigma^{-1}
\big(
\alpha_{r,\nu}^{(i)}(\sigma^{-1})
\big)
 \chi(\sigma)
\phi_\nu.
\]
In this way every element $x\in F_{g-2-i}$ is acted on by $\sigma$ in terms of the action 
\[
\sigma \star x=  h\big(^\sigma h^{-1}(x) \big).
\]
On the other hand the elements $h(\phi_r)$ are in $F_{g-2-i}$ and by lemma \ref{star-dual-action}
there is an element $\sigma'\in G$ such that 
\[
\sigma'h(\phi_r)=
\sum_{\nu=1}^r
\alpha_{\nu,r}^{(g-2-i)}(\sigma') h(\phi_\nu). 
\]
Since the element $\phi_\nu$ has maximal degree among generators of $F_i$ the element $h(\phi_r)$ has minimal degree. This means that all coefficients 
\[
\alpha_{\nu,r}^{(g-2-i)}(\sigma')=
\sigma
\big(
\alpha_{r,\nu}^{(i)}(\sigma^{-1})
\big)
 \chi(\sigma)
\]
are zero for all $\nu$ such that $\deg m_\nu < \deg \mu_r$. Therefore 
 all coefficients $a_{\nu,r}^{(i)}(\sigma)$
for $\nu$ such that $\deg m_\nu < \deg m_r$ are zero. This holds for all $\sigma\in G$.  
By considering in this way all elements $\phi_{r-1},\phi_{r-2},\ldots,\phi_1$, which might have greater degree than the degree of $\phi_r$ the result follows. 
\end{proof}

%---------------------------
\section{Representations on the free resolution}
%---------------------------

Each $S$-module $F_i$ in the minimal free resolution can be seen as a series of representations of the group $G$. Indeed, the modules $F_i$ are graded and there is an action of $G$ on each graded part, given by representations
\[
\rho_{i,d}:G \rightarrow \mathrm{GL}\big( F_{i,d}\big),
\]
where $F_{i,d}$ is the degree $d$ part of the $S$-module $F_i$. 
In equation (\ref{TorRep}) we have already considered the representation on $\mathrm{GL}(\mathrm{Tor}_i^S(k,S_X))$ which is $\rho_{i,d}$. 
Proposition \ref{prop:respect-degrees} shows us that there is a decomposition 
\[
\mathrm{Tor}_i^S(k,S_X)=\bigoplus_{j \in \mathbb{Z}}
\mathrm{Tor}_i^S(k,S_X)_j,
\]
where $\mathrm{Tor}_i^S(k,S_X)_j$ is the $k$-vector space generated by generators of $F_i$ that have  degree $j$. This is a vector space of dimension $b_{i,j}$.  

Denote by $\mathrm{Ind}(G)$ the set of isomorphism classes of indecomposable $k[G]$-modules. 
If $k$ is of characteristic $p>0$ and $G$ has no-cyclic $p$-Sylow subgroup then the set $\mathrm{Ind}(G)$ is infinite, see \cite[p.26]{MR2267317}. 
Suppose that each $\mathrm{Tor}_i^S(k,S_X)_j$ admits the following decomposition in terms of $U\in \mathrm{Ind}(G)$:
\[
\mathrm{Tor}_i^S(k,S_X)_j=\bigoplus_{U\in \mathrm{Ind}(G)}
 a_{i,j,U} U \text{ where } a_{i,j,U} \in \Z.
\]
We obviously have that 
\[
b_{i,j}=\sum_{U \in \mathrm{Ind}(G)} a_{i,j,U} 
\dim_k U. 
\]
The $G$-structure of $F_i$ is given by 
\[
\mathrm{Tor}_i^S(k,S_X) \otimes S,
\]
that is the $G$-module structure of $F_{i,d}$ is given by 
\[
F_{i,d}=\bigoplus_{d\in \mathbb{Z}} 
\bigoplus_{j\in \mathbb{Z}}
\mathrm{Tor}_i^S(k,S_X)_{d-j} \otimes S_j.
\]
Notice that structure of $S_d$ can be expressed in terms of the decomposition of $H^0(X,\Omega_X)$ in terms of indecomposable modules using the operations in the Grothendieck ring of $G$.

\section*{Appendix}
In what follows we will study the canonical ideal of the Fermat curve, for $n\geq 6$,following the method developed in \cite{1905.05545} and we will prove proposition \ref{canGenerators}.

Observe that the holomorphic differentials given in eq. (\ref{Hol-Fermat}) are in 1-1 correspondence with the elements of the set
$\mathbf{A}=\{ (i,j): 0 \leq i+j \leq n-3\}\subset \mathbb{N}^2$.
First we introduce the following term order on the polynomial algebra $S:=\mathrm{Sym} H^0(X,\Omega_X)$.
\begin{definition}\label{term-order}
Choose any term order $\prec_t$ for the variables $\left\{\omega_{N,\mu}:(N,\mu)\in A \right\}$ and define the term order $\prec$ on the monomials of $S$ as follows:
\begin{equation}\label{term-order1}
\omega_{N_1,\mu_1}\omega_{N_2,\mu_2}\cdots \omega_{N_d,\mu_d}\prec \omega_{N'_1,\mu'_1}\omega_{N'_2,\mu'_2}\cdots \omega_{N'_s,\mu'_s}\text{ if and only if}
\end{equation}
\begin{itemize}
\item $d<s$ or\\
\item $d=s$ and $\sum \mu_i >\sum \mu'_i$ or\\
\item $d=s$ and $\sum \mu_i =\sum \mu'_i$  and $\sum N_i <\sum N'_i$\\
\item $d=s$ and $\sum \mu_i =\sum \mu'_i$  and $\sum N_i =\sum N'_i$ and
\[\omega_{N_1,\mu_1}\omega_{N_2,\mu_2}\cdots \omega_{N_d,\mu_d}\prec_t \omega_{N'_1,\mu'_1}\omega_{N'_2,\mu'_2}\cdots \omega_{N'_s,\mu'_s}.\]
\end{itemize}
\end{definition} 

T
\begin{lemma}
\label{count-sum}
The number of natural numbers $0 \leq i,j$ such that 
$0 \leq i+j \leq E\in \mathbb{N}$ equals $(E+1)(E+2)/2$.
\end{lemma}
\begin{proof}
Evaluate $\sum_{i=0}^{E} \sum_{j=0}^{E-i} 1$. 
\end{proof}
 We will use the following lemma
\begin{lemma}\label{lemma 1} 
Let $J$ be the ideal generated by the elements $G_1,G_2$ and let $I$ be the canonical ideal. Assume that the cannonical is generated by elements of degree $2$.  If 
$\dim_L \left(S/\init (J)\right)_2\leq 3(g-1)$, then $I= J$.
\end{lemma}

We extend the correspondence between the variables $\omega_{i,j}$ and the points of $\mathbf{A}$ to a correspondence between monomials in $S$ of standard degree $2$ and points of the Minkowski sum of  $\mathbf{A}$ with itself, defined as
\begin{equation}\label{Minkowski-def}
\mathbf{A}+\mathbf{A}=\{(i+i',j+j')\;|\;(i,j),(i',j')\in \mathbf{A}\}\subseteq\mathbb{N}^2.
\end{equation}

\begin{proposition}\label{mdeg-characterization}
Let $\mathbf{A}$ be the set of exponents of the basis of holomorphic differentials, and let $\mathbf{A}+\mathbf{A}$ denote the Minkowski sum of $\mathbf{A}$ with itself, as defined in (\ref{Minkowski-def}). Then
\[(\rho,T)\in 
\mathbf{A}+\mathbf{A}
\Leftrightarrow\exists \;\omega_{i,j}\omega_{i',j'}\in S\text{ such that } \mdeg(\omega_{i,j}\omega_{i',j'})=(2,\rho,T).\]
\end{proposition}

For each $n\in\mathbb{N}$ we write $\mathbb{T}^n$ for the set of monomials of degree $n$ in $S$ and  proceed with the characterization of monomials that do not appear as leading terms of binomials in
  $G_1\subseteq J$.
\begin{proposition}\label{proposition 2} 
Let $\sigma$ be the map of sets 
\begin{eqnarray*}
\sigma:\mathbf{A}+\mathbf{A}
&\rightarrow& \mathbb{T}^2\\
(\rho,T)&\mapsto& \min_{\prec}\{\omega_{i,j}\omega_{i',j'}\in \mathbb{T}^2\;|\; (\rho,T)=(i+i',j+j')\}
\end{eqnarray*}
Then
\[\sigma(\mathbf{A}+\mathbf{A})=
\{\omega_{i,j}\omega_{i',j'}\in \mathbb{T}^2\;|\;\omega_{i,j}\cdot\omega_{i',j'}\neq\init (f),\;\forall\; f\in G_1
\}\]
\end{proposition}

The above proposition gives a characterization of the monomials that do not appear as initial terms of elements of $G_1$.  However, some of these monomials appear as initial terms of polynomials in $G_2$:
\begin{proposition}\label{proposition 3}
Let
\[C=\{(\rho,b)\in \mathbf{A}+\mathbf{A}\;|\; 
\rho=n+a, 0\leq a+b \leq n-6, a,b\in \mathbb{N}\}
\]
%  \in  \mathbf{A}+\mathbf{A} \text{ and }
% (\rho-n,T) \in \mathbf{A}+\mathbf{A}
% \}.\]
Then
 \[\sigma(C)\subseteq \{\omega_{i,j}\omega_{i',j'}\in \mathbb{T}^2\;|\;\exists\;g\in G_2 \text{ such that }\omega_{i,j}\omega_{i',j'}=\init (g)\}
 \]
 Moreover $\#C=\#\sigma(C)=(n-5)(n-4)/2$.
\end{proposition}
\begin{proof}
By the form of the equations of $G_2$ we have that 
\[
i_1+i_2=n+a=\rho, j_1+j_2=b
\]
so 
\[
i_3+i_4=a=\rho-n, j_3+j_4=n+b, i_5+i_6=a=\rho-n, 
j_5+j_6=b=T
\]
We should have 
\begin{align*}
0 \leq a+b \leq n-6
\end{align*}
and by lemma \ref{count-sum} we have that the cardinality of $C$ equals $(n-5)(n-4)/2$. 
\end{proof}
We now compute that 
\[
3(g-1)-(\#(\mathbf{A}+\mathbf{A})-C)=
3\left(
% \frac{(n-2)(n-1)}2
g-1
\right)
-\frac{(2n-5)(2n-4)}2 + \frac{(n-5)(n-4)}2=0,
\]
so by lemma \ref{lemma 1} we have that $I=J$.

 \def\cprime{$'$}

\end{document}